\documentclass[12pt,a4paper]{article}

\usepackage{cite}

 \usepackage{caption}
\usepackage{subcaption}
\usepackage{float}
\usepackage{rotating}
\usepackage{stackengine}
\usepackage{amsmath}
\usepackage{amssymb}
\usepackage{amsthm}
\usepackage{amsfonts}
\usepackage{graphicx}
\usepackage[all]{xy}
\usepackage{verbatim}
\usepackage[left=2cm,right=2cm,top=3cm,bottom=2.5cm]{geometry}
\usepackage[vcentermath]{youngtab}
\usepackage{ytableau}

\DeclareMathOperator\Inf{\rm Inf}
\DeclareMathOperator\Sym{\rm Sym}
\DeclareMathOperator\SYT{\rm SYT}

\theoremstyle{definition}
\newtheorem{theorem}{Theorem}
\newtheorem*{theorem*}{}

\newtheorem{lemma}[theorem]{Lemma}
\newtheorem{definition}[theorem]{Definition}

\newtheorem{question}[theorem]{Question}

\newtheorem{corollary}[theorem]{Corollary}
\begin{document}
	
\title{ Some Properties of Generalized Foulkes Module} \author{P\'{a}l Heged\"{u}s \and Sai Praveen Madireddi }
\maketitle
\begin{abstract}
Describing the decomposition of Foulkes module $F_b^a$ into irreducible Specht modules is an open problem for $a,b > 3$. In this article we provide a new approach for the Generalized Foulkes module $F_{\nu}^a$ (with arbitrary partition $\nu$ of $b$) through its restriction to a maximal Young subgroup 
${S_b \times S_{ab -b}}$.
\end{abstract}
\section{Introduction}
The modules in this paper are defined over the complex numbers. For $a,b>1$ integers, let $n=ab$. The Foulkes module $F_{b}^{a}$ is the permutation module of $S_{n}$ acting on the set of partitions of type $(a^b)$, that is on partitions of $\{1, 2, \ldots, n\}$ into $b$ sets of size $a$ each. 


The simple $S_b$-modules are parametrised by the partitions of $b$, the simple module corresponding to a partition $\nu\vdash b$ is the so-called Specht module, $S^\nu$. In particular, $S^{(b)}$ is the trivial, while $S^{(1^b)}$ is the sign module of $S_b$. The wreath product $S_a\wr S_b\leq S_{n}$ has a normal subgroup $S_a\times S_a\times\cdots \times S_a$ ($b$ factors), called the base group, with factor group isomorphic to $S_b$, hence we may consider $S^\nu$ as an $S_a\wr S_b$-module with kernel containing the base group. This module is the inflated Specht module, denoted by $\Inf_{S_b}^{S_a \wr S_b} S^{\nu}$ is a simple modele of $S_a \wr S_b$. The $\nu$-generalized Foulkes module is the induced module of this inflation to $S_{n}$, in formula, $F_\nu^a = \Inf_{S_b}^{S_a \wr S_b} S^{\nu} \uparrow^{S_{n}}$. When $\nu = (b)$, we recover the original Foulkes module, $F_{(b)}^a=F_b^a$. If, however, $a=1$ then $S_a\wr S_b=S_{b}$ and $F_\nu^1=S^\nu$.

Thrall \cite{TRM} decomposed the Foulkes module into simple components for $a = 2$ and for $b = 2$:
\begin{equation*}
	F_b^2 = \bigoplus_{\lambda \vdash b} S^{2\lambda};\qquad
	F_2^a = \bigoplus_{\substack{\lambda \vdash 2a\\ \lambda_1+\lambda_2=2a }}  S^{\lambda}.
\end{equation*}

Foulkes \cite{HOF} conjectured that if $a \leq b$ then $F_b^a$ can be embedded in $F_a^b$, for $a=2$ this is an immediate consequence of Thrall's result.
Dent \cite{SCD} decomposed the Foulkes module for $F_3^a$ and $F_a^3$ into Specht modules and verified the conjecture for $a =3$. For $a>3$ no full decomposition of the Foulkes module is known, albeit the conjecture is proven for $a\leq 5$, see \cite{TM} \cite{MN} \cite{CIM}.
For an integer $0<k<n$ we define $\Omega_k$ as the set of partitions of $k$ which are subpartitions of $(a^b)$, that is all parts are of size less than or equal to $a$. Then the restriction of the generalized Foulkes module to $S_k\times S_{n-k}$ has a natural decomposition (see below, Definition~\ref{def:component}) indexed by $\Omega_k$:
\begin{equation}\label{eq:Omega-indexing}
		F_{\nu}^a \downarrow_{S_k \times S_{n-k} } = \bigoplus_{\lambda \in \Omega_k} V_{\nu, a}^{\lambda}
\end{equation}
We are concerned mainly with the $(1^k)$-component $U_{\nu,a}=V_{\nu,a}^{(1^k)}$.

The main theorem of this paper is the following. Let $\mu^{\perp}$ denote the conjugate of the partition $\mu$. In particular, $(1^b)^\perp=(b)$. For $\mu,\lambda,\nu\vdash b$ let $c_{\mu,\lambda}^\nu=c_{\lambda,\mu}^\nu$ denote the Kronecker coefficient, that is the multiplicity of $S^\nu$ in the tensor product $S_b$-module $S^\mu\otimes S^\lambda$. For $0<k<n$  any pair of $S_k$-module $M$ and $S_{n-k}$-module $N$ defines an $S_k\times S_{n-k}$-module $M\times N$. The simple modules of $S_k\times S_{n-k}$ are the ones  $S^\mu\times S^\lambda$ coming from pairs of Specht modules.
\begin{theorem}	\label{thm:main}
	Let $k=b$ and as above, $U_{\nu,a}=V_{\nu,a}^{(1^b)}$.
	 Then 
	
	\begin{equation}\label{eq:U_a}
	U_{\nu, a} \cong \bigoplus_{\mu,\lambda\vdash b} c_{\mu, \lambda}^{\nu} S^{\mu} \times F_{\lambda}^{a-1}.
	\end{equation}
	In particular, for $a=2$
	\begin{equation}\label{eq:U_2}
	U_{\nu, 2} \cong \bigoplus_{\mu,\lambda\vdash b} c_{\mu, \lambda}^{\nu} S^{\mu} \times S^{\lambda}.
	\end{equation}
	
\end{theorem}
As noted above, $S^{(b)}$ is the trivial, while $S^{(1^b)}$ is the sign module of $S_b$. So $c_{(b),\nu}^\nu=1$ and $c_{(1^b),\nu^\perp}^\nu=1$. Hence the multiplicity of both $S^{(b)} \times S^{\nu}$ and $S^{(1^b)} \times S^{\nu^{\perp}}$ in $U_{\nu, 2}$ are $1$.
Two important special cases of the main theorem are the following.
\begin{corollary}\label{cor:Foulkes}
Let $a,b \in \mathbb{N}$.
The $(1^b)$-summand of the Foulkes module $F_{(b)}^a$ restricted to $S_b\times S_{n-b}$ is
\begin{equation*}
U_{(b), a} \cong \bigoplus_{\lambda \vdash b} S^{\lambda} \times F_\lambda^{a-1}
\end{equation*}
In particular, for $a = 2$
\begin{equation*}
U_{(b), 2} \cong \bigoplus_{\lambda \vdash b} S^{\lambda} \times S^{\lambda}
\end{equation*}

\end{corollary}
\begin{corollary}\label{cor:gen_Foulkes}
Let $a,b \in \mathbb{N}$. The $(1^b)$-summand of the generalized Foulkes module $F_{(1^b)}^a$ restricted to $S_b\times S_{n-b}$ is
\begin{equation*}
U_{(1^{b}), a} \cong \bigoplus_{\lambda \vdash b} S^{\lambda^\perp} \times F_\lambda^{a-1}.
\end{equation*}
In particular, for $a = 2$
\begin{equation*}
U_{(1^b), 2} \cong \bigoplus_{\lambda \vdash b} S^{\lambda^\perp} \times S^{\lambda}
\end{equation*}
\end{corollary}

With the help of the so-called semistandard homomorphism de Boeck\cite{DBM} proved that the multiplicity of $S^{\lambda +(b)}$ in $F_{b}^{a+1}$ is at least the multiplicity of $S^{\lambda}$ in $F_b^a$ and the multiplicity of  $S^{\lambda +(1^b)}$ in $F_{b}^{a+1}$ is equal to the multiplicity of $S^{\lambda}$ in $F_{(1^b)}^a$. This we establish as a corollary of our results. See Corollary \ref{cor:gen_foulkes_min}.
\section{Preliminaries}
Our main reference for the representations of the symmetric group is \cite{GJamesRTSG}, recall especially the notions of the $\nu$-tableau $t$, the tabloid $\{t\}$ and the polytabloid $e_t$. We call $\SYT(\nu)$ the set of standard $\nu$-tableaux. 

Let $a,b \geq 2$ fixed integers and $n=ab$. Denote by $H=H_b^a$ the set of ordered partitions of $\{1,\,2,\ldots, n\}$ into $b$ sets of size $a$ each. Clearly, $S_n$ acts on $H$ by permuting the letters.  But $S_b$ also acts on $H$ by permuting the indices, that is, for $X=(X_{1}, \ldots, X_{b})\in H_a^b$ we have $\sigma X = (X_{\sigma(1)}, \ldots, X_{\sigma(b)})$. Let $I=I_b^a$ be a set of representatives of $S_b$ orbits. Therefore $|H|=n!/(a!)^b$ and $|I|=n!/(a!)^b b!$.


\begin{definition}\label{def:t_X}
For a partition $\nu\vdash b$, a $\nu$-tableau $t$, 
and  $X\in H$ let $t_X$ be the $\nu$-shaped diagram with $X_l$ replacing $l$ in $t$. 
Similarly, $\{t_X\}$ is the ``tabloid'' where each $l$ in $\{t\}$ is replaced by $X_l$. If $g\in  S_n$ then clearly 
$gt_{X}=t_{gX}$ and $g\{t_X\}=\{t_{gX}\}$. Finally, a $\nu$-polytabloid, $e_t$ is a certain element of the vector space with the $\nu$-tabloids being a formal basis. The corresponding $\nu$-``polytabloid'' $e_{t_X}$, is an element of the vector space $V$ having the $\nu$-``tabloids'' as a formal basis. Here again, the parts of $X$ are replacing the letters. We again have that for $g\in S_n$, $ge_{t_X}=e_{t_{gX}}$.
\end{definition}

As an example let $\nu = (2,2)$. 
For
\begin{center}
     $t$ =     \begin{ytableau}   
               1 & 2 \\                         
               3& 4                            
	           \end{ytableau}\,,                      
           \qquad
           $t_{X}$ = \begin{ytableau}
  	                      X_1 & X_2\\
  	                      X_3 & X_4
                          \end{ytableau}\,,
            \qquad
            $\{t\} =\ytableausetup                     
            {boxsize = 1.4em,
            tabloids, mathmode}
            \ytableaushort{       
 	        12,34
            }$\,,
            \qquad
            $\{t_X\} =\ytableausetup                     
            {boxsize = 1.4em,
            tabloids, mathmode}
            \ytableaushort{       
 	        {X_1} {X_2}, {X_3} {X_4}
            }$\,.
\end{center}

Let $a=3,\,b=4$ and \[X = (\{1,2,3\},\{4,5,6\},\,\{7,8,9\},\,\{10, 11, 12\})\in H_3^4.\] Then

\[e_t = \ytableausetup                     
        {boxsize = normal, tabloids}           
        \ytableaushort{       
        12, 34
        }
        -
        \ytableausetup
        {boxsize = normal, tabloids}
        \ytableaushort{
        14, 32
        }
        +
        \ytableausetup
        {boxsize = normal, tabloids}
        \ytableaushort{
        34, 12
        }
        -
        \ytableausetup
        {boxsize = normal, tabloids}
        \ytableaushort{
        32, 14
        }\]
 and

 \begin{align*}
 e_{t_X} =&\ytableausetup                     
            {boxsize = 1.4em,
            tabloids, mathmode}
            \ytableaushort{       
 	        \{123\}\{456\}, \{789\}\{{10} {11} {12}\}
            }
            -
            \ytableaushort{
          	 \{123\}\{{10} {11} {12}\}, \{789\}\{456\}
            }\\ \\
           + &
            \ytableaushort{
            \{789\}\{{10} {11} {12}\}, \{123\}\{456\}
            }
            -
            \ytableaushort{
 	        \{789\}\{456\}, \{123\}\{{10} {11} {12}\}
            }
\end{align*}

The set of standard  $\nu$-polytabloids, $B_{\nu}$ = $\{ e_t \mid t\in\SYT(\nu)\}$ forms a basis of the Specht module $S^{\nu}$.

\begin{lemma}\label{lem:polytabloids}The inflation $\Inf_{S_b}^{S_a \wr S_b} S^{\nu}$ has basis  $B_{\nu,X} = \{ e_{t_X} \mid t\in\SYT(\nu)\}$ in the vector space $V$ of $\nu$-``polytabloids.''
\end{lemma}
Before the proof we remark that the role of $X$ in the definition is to fix the wreath product, which acts on the parts of $X$. 

\begin{proof}
For $g\in S_a\wr S_b$ the image in $S_b\cong S_a\wr S_b/S_a^b$ is denoted by $\tau_g$.
Now $\tau_g\in S_b$ acts on the entries of $t$ and on the indices of $X$ and we have $(\tau_g t)_X=t_{\tau_g X}$ and therefore $ge_{t_X}=e_{(\tau_g t)_X}=e_{t_{\tau_g X}}=e_{t_{gX}}$. The action of $S_b$ on $S^\nu$ with respect to the basis $B_\nu$ is indeed inflated to the action of $S_a\wr S_b$ with respect to the basis $B_{\nu,X}$. The simple module $\Inf_{S_b}^{S_a \wr S_b} S^{\nu}$ is generated by $e_{t_X}$ for any tableau $t$ like $S^{\nu}$ is generated by any $e_t$.
\end{proof}

\begin{lemma} \label{lem:arbpoly}
Fix a partition $\nu\vdash b$ and a $\nu$-tableau $t$. Let $V_0=\langle e_{t_X}\mid X\in H\rangle\leq V$. Then $V_0$ is $S_n$-invariant and as an $S_n$-module it is isomorphic to the generalized Foulkes module $F_{\nu}^ a = \Inf_{S_b}^{S_a \wr S_b} S^{\nu} \uparrow^{S_{n}}$. Further, $B_{\nu}^a = \{e_{s_Z} \mid Z \in I, s \in\SYT(\nu)\}$ is a basis of $V_0\cong F_{\nu}^a$ such that if $t$ is a $\nu$-tableau, $X\in H$ and \[e_{t_X} = \sum_{e_{s_Z}\in B_\nu^a} c_{s,Z} e_{s_{Z}}\]
then $c_{s,Z}=0$ unless $X$ is in the $S_b$-orbit of $Z$.
\end{lemma}
\begin{proof}As $ge_{t_X}=e_{t_{gX}}$, $V_0$ is $S_n$-invariant, so an $S_n$-module. Also, $V_0$ is generated by $e_{t_X}$ for any $X$. Let us fix a $W=S_a \wr S_b\leq S_n$ and $Y\in H$ the corresponding partition. As $e_{t_Y}$ generates $V_0$ as an $S_n$-module, by \cite[Corollary~$8.3$]{JLAlperin} it is enough to confirm that
   \[\dim V_0=|S_n:W||B_{\nu,Y}|=\frac{n!}{(a!)^b b!}|B_\nu|=|I||\SYT(\nu)|.\] 

	Pick $Z \in I$ in the $S_b$-orbit of $Y$. If $Z = \sigma Y$ for a $\sigma\in S_b$ then $e_{t_Y}=e_{t_{\sigma Z}}=e_{(\sigma t)_Z}
$. If
	\[ e_{\sigma t} = \sum_{s \in\SYT(\nu)} c_{s} e_{s}
 \]
	then
	\begin{equation}\label{eq:unique_decomp}
	    e_{t_Y}=e_{(\sigma t)_{Z}} = \sum_{s \in\SYT(\nu)} c_{ s} e_{s_{Z}}.
	\end{equation} 
Suppose that 
\[0=\sum_{\substack{Z\in I\\ s \in\SYT(\nu)}} d_{ s,Z} e_{s_{Z}}=\sum_{Z\in I}\sum_{s \in\SYT(\nu)}d_{s,Z}e_{s_Z}.\] If $Z\ne Z^\prime$ (so not in the same orbit) and $s,t\in\SYT(\nu)$ arbitrary  
then the ``tabloids'' occurring in $e_{s_Z}$ and $e_{t_{Z^\prime}}$ are distinct, so we must have \[
0=\sum_{s \in\SYT(\nu)}d_{s,Z}e_{s_Z},\,\forall Z\in I.
\] But $B_{\nu,Z}$ is a basis whence 
$0=d_{s,Z}$ for every $s\in\SYT(\nu)$ and $Z\in I$. 

Therefore $B_{\nu}^a = \{e_{s_Z} \mid Z \in I, s \in\SYT(\nu)\}$ is indeed a basis of $V_0$, $\dim V_0=|B_{\nu}^a|=|I||\SYT(\nu)|$ and thus \eqref{eq:unique_decomp} is a unique expression so the last part of the Lemma also holds.
\end{proof}


Let $1<k<ab$. We describe a decomposition of the restriction of the generalized Foulkes module to a maximal intransitive subgroup $
{S_k \times S_{n -k}}$. 
As above, $\Omega_k$ is the set of those partitions of $k$ that are subpartitions of $(a^b)$.
In the following notation the dependence on $a$ is generally suppressed.
\begin{definition}\label{def:component}
For $\lambda\in \Omega_k$ let\[P_{\lambda} =P_\lambda^a= \{ X\in H_b^a\mid\lambda\text{ is the partition type of } \{1,2,\ldots,k\} \cap X\}\] and let the \emph{$\lambda$-component}, $V_{\nu, a}^{\lambda}$ be the $S_k \times S_{n-k}$-module generated by the $S_k\times S_{n-k}$-invariant set $\{e_{t_{X}} \mid X \in P_{\lambda} \}$ for any $t$ of shape $\nu$.
\end{definition}
Note that $X\in P_\lambda$, $\sigma\in S_b$ implies $\sigma X\in P_\lambda$, so $I\cap P_\lambda$ is a set of representatives of $S_b$-orbits of $H$ that lie in $P_\lambda$. The set $\{ e_{t_{X}} \mid X \in I\cap P_{\lambda},\, t\in\SYT(\nu)\}$ is a basis of $V_{\nu, a}^{\lambda}$. Indeed, the vector space they generate is $S_k\times S_{n-k}$-invariant and
\[\bigcup_{\lambda\in\Omega_k} \{ e_{t_{X}} \mid X \in I\cap P_{\lambda},\, t\in\SYT(\nu)\}=B_\nu^a.
\]
The mentioned decomposition is thus
\begin{equation} 
	F_{\nu}^a \downarrow_{S_k \times S_{n-k} } = \bigoplus_{\lambda \in \Omega_k} V_{\nu, a}^{\lambda}.
\end{equation}

Here comes an example with $a=3,\,b=k=4$ and $\nu = (2^2)$. 
We have $\Omega_4= \{(3,1),\, (2^2),\, (2,1^2),\, (1^4)\}$. Then $V=V_{(2^2), 3}^{(2, 1^2)}$ is generated by the set $\{ e_{t_{X}} \mid X \in P_{(2, 1^2)},\, t\in\SYT((2^2)) \}$. Let $t$ be as before and let \[X = (\{1,2,5\}, \{3,7,6\}, \{4,8,9\},\{10, 11, 12\}),\,X\cap \{1,2,3,4\}=(\{1,2\},\{3\},\{4\}).\]
 \begin{align*}
e_{t_X} =& \ytableausetup                     
{boxsize = 1.4em, tabloids, mathmode}           
\ytableaushort{       
	\{125\}\{367\}, \{489\}\{{10} {11} {12}\}
}
-
\ytableaushort{
	\{125\}\{{10} {11} {12}\}, \{489\}\{367\}
}\\\\
+ &
\ytableaushort{
	\{489\}\{{10} {11} {12}\}, \{125\}\{367\}
}
-
\ytableaushort{
	\{489\}\{367\}, \{125\}\{{10} {11} {12}\}
}
 \end{align*}
For $g = (1 4)(5 6 7)\in S_4\times S_{12}$
\begin{align*}
g e_{t_X} =& \ytableausetup                     
{boxsize = 1.4em, tabloids, mathmode}           
\ytableaushort{       
	\{246\}\{357\}, \{189\}\{{10} {11} {12}\}
}
-
\ytableaushort{
	\{246\}\{{10} {11} {12}\}, \{189\}\{357\}
}\\\\
+ &
\ytableaushort{
	\{189\}\{{10} {11} {12}\}, \{246\}\{357\}
}
-
\ytableaushort{
	\{189\}\{357\}, \{246\}\{{10} {11} {12}\}
}
\end{align*}
\section{Properties of $U_{\nu, a}$}
Here we focus on $k=b$ and especially on $U_{\nu, a} = V_{\nu, a}^{(1^b)}$. Recall from the discussion after Definition~\ref{def:component} that $U_{\nu,a}$ has basis $\{ e_{t_{X}} \mid X \in I_b^a\cap P_{(1^b)},\, t\in\SYT(\nu)\}$. To prove  Theorem~\ref{thm:main} we first deal with the $a=2$ case and then connect it to the arbitrary $a>2$ case.
\begin{definition}\label{def:Ttau}
Let $t$ be a $\nu$-tableau and $\tau \in \Sym(\{b+1, \ldots, 2b\})$. We define the ordered partition
$T(\tau)=(\{1, \tau(b+1)\}, \{2, \tau(b+2)\}, \dots, \{b, \tau(2b)\})\in H_b^2\cap P_{(1^b)}^2$.

\end{definition}
It is clear that the
set $\{T(\tau) \mid \tau \in \Sym(\{b+1, \ldots, 2b\})\}$ is a full set of representatives of the $S_b$ orbits of $H_b^2\cap P_{(1^b)}^2$. So $\{e_{t_{T(\tau)}} \mid \tau \in \Sym(\{b+1, \ldots, 2b\}),\, t\in\SYT(\nu)\}$ is a basis of $U_{\nu,2}$.

\begin{lemma}\label{lem_arb:a}
	Let $a, b \in \mathbb{N}$, $\nu$ a partition of $b$. Then
	\begin{equation}\label{lem:arb:a}
		U_{\nu, a} \cong \Inf_{S_b \times S_b}^{S_b \times (S_{a-1} \wr S_b)} U_{\nu, 2} \uparrow^{S_b \times S_{n-b}}
	\end{equation}
\end{lemma}	
\begin{proof}
Fix a $\nu$-tableau $t$.
For $Y = \{Y_1,\ldots, Y_b\}  \in H_b^{a-1}$ (where the underlying set is $\{b+1,\,b+2,\ldots,n\}$) and $\sigma\in S_b$ let $Y_{\sigma}$ be the ordered partition 
$(\{\sigma(1)\}\cup Y_1, \dots,  \{\sigma(b)\}\cup Y_b)\in H_b^a$. 
Note that $Y_{\sigma}\in H_b^a$ and $\sigma Y\in H_b^{a-1}$ are different. As is Definition~\ref{def:Ttau}, for $\sigma=id$ put $T(Y)=Y_{id}$. 
Separating the smallest element of each part for $X\in P_{(1^b)}$ we see that 
$P_{(1^b)}=\{ Y_{\sigma} \mid Y \in H_b^{a-1}, \sigma \in S_b \}$. Similarly, $P_{(1^b)}\cap I_b^a=\{ T(Y) \mid Y \in H_b^{a-1}\}$. The module $U_{\nu, a}$ is the $S_b\times S_{n-b}$-module generated by the $S_b\times S_{n-b}$-invariant set $\{e_{t_X}  \mid X \in P_{(1^b)}\}=
\{ e_{t_{Y_{\sigma}}} \mid Y \in H_b^{a-1}, \sigma \in S_b \}$ for any $t$ of shape $\nu$. Its basis is $\{ e_{t_{T(Y)}} \mid Y \in H_b^{a-1}, t\in\SYT(\nu)\}$.

Fix $Y\in H_b^{a-1}$ and with it a wreath product $S_{a-1}\wr S_b\leq S_{n-b}$. As before, for $g\in S_{a-1}\wr S_b$ the image in $S_b\cong S_{a-1}\wr S_b/S_{a-1}^b$ is denoted by $\tau_g$. Now for any $\sigma\in S_b$, and $g\in S_{a-1} \wr S_b$  we get $(gY)_\sigma=(\tau_g Y)_\sigma
= (\{\sigma(1)\}\cup Y_{\tau_g(1)}, \ldots, \{\sigma(b)\}\cup Y_{\tau_g(b)})$. Denote by $W$ the $S_b \times (S_{a-1} \wr S_b)$-module  generated by the set $\{ge_{t_{Y_\sigma}}\mid \sigma\in S_b,\,g\in S_{a-1} \wr S_b\}=\{e_{t_{(\tau Y)_\sigma}} \mid  \tau,\sigma \in S_b\}$. Using the argument of Lemma~\ref{lem:polytabloids} we obtain that $W$ is the inflation of $U_{\nu,2}$ to $S_b\times S_{a-1}\wr S_b$ and its basis is $\{e_{t_{T(\tau Y)}} \mid \tau \in S_b,\, t \in\SYT(\nu)\}$.

The argument of Lemma~\ref{lem:arbpoly} now provides
\begin{equation}
  U_{\nu, a} \cong W \uparrow^{S_b \times S_{n-b}}, 
\end{equation}because the induction takes place only in the second component.
\end{proof}

Thus the study of $U_{\nu, 2}$ might give some interesting information on  the generalized Foulkes module.

Now we are ready to prove our main theorem.
\begin{proof}[Proof of Theorem~\ref{thm:main}]
Denote by $G_1=\Sym\{1,\ldots,b\} $ and $G_2=\Sym{\{b +1,\ldots, 2b\}}$, both isomorphic to $S_b$. 
Recall that a basis of $U=U_{\nu, 2}$ is 
$\{e_{t_{T(\tau)}} \mid \tau \in G_2,\,t\in\SYT(\nu)\}$.

We determine the values of the $G_1\times G_2$-character $\chi$ of $U$ which then helps us to identify the decomposition of $U$ into irreducible modules.

Let $\chi^{\nu}$ denote the irreducible character of the Specht module $S^{\nu}$. We claim that for $(g_1,g_2)\in G_1\times G_2$
\[
\chi(g_1,g_2)=\begin{cases}0,&\text{if } g_1,g_2\text{ are of different cycle structure;}\\
 |C_{S_b}(g_1)|\chi^\nu(g_1),&\text{if } g_1,g_2\text{ are of the same cycle structure.}\end{cases}
 \]

Since  the character value is the sum of the coefficients of basis element $e_{t_{T(\tau)}}$ in $(g_1, g_2)e_{t_{T(\tau)}}$ we need to compute them in order to prove the claim. 
Let $h=(1,b+1)(2,b+2)\cdots(b,2b)$ and $g_3=hg_1h$ be the shifted permutation to $\{b+1,\ldots ,2b\}$, that is $g_3$ sends $b + l$ to $b+k$ if and only if $g_1$ sends $l$ to $k$. In particular, $g_1$ and $g_3$ have the same cycle structure. Clearly, $(g_1, g_2) e_{t_{T(\tau)}} = e_{(g_1t)_{T(g_2 \tau g_3^{-1})}}= e_{s_{T(\varrho)}}$ where $s=g_1t$ and $\varrho = g_2 \tau g_3^{-1}\in G_2$. Note that $s=g_1t$ need not be standard so  $e_{s_{T(\varrho)}}$ might not be a basis element! 
However, if $g_1 e_t = e_s = \sum d_{r} e_{r}$ then
\begin{equation}
(g_1,g_2) e_{t_{T(\tau)}} = e_{s_{T(\varrho)}} = \sum d_{r} e_{r_{T(\varrho)}}.
\end{equation}
Therefore the coefficient of $e_{t_{T(\tau)}}$ in $e_{s_{T(\varrho)}}$ is $0$ unless $T(\tau) = T(\varrho)$. In the latter case $\tau=\rho$. 
Hence $g_2=\tau g_3\tau^{-1}$ and $g_1=hg_3h$ must be conjugate, so of the same cycle structure.

From now let $g_1$ and $g_2$ be fixed and of the same cycle structure. 
For computing the character value we need to find the number of permutations $\tau$ and standard tableau $t$ such that $e_{t_{T(\tau)}}$ is a linear summand of $(g_1,g_2) e_{t_{T(\tau)}} = e_{s_{T(\varrho)}}$. This can only happen if $e_t$ is a linear summand of $e_s = g_1 e_t$ and in that case $\varrho = g_2 \tau g_{3}^{-1} = \tau$, in other words
$\tau^{-1} g_2 \tau = g_{3}$. The number of such $\tau$ is equal to the order of the centralizer $|C_{S_b}(g_2)|=|C_{S_b}(g_1)|$ which proves the claim. 

Let $k_{g_1}=b!/|C_{S_b}(g_1)|$ denote the size of conjugacy class of $g_1$. Then the multiplicity $d_{\mu, \lambda}^{\nu}$ of the simple module $S^{\mu} \times S^{\lambda}$ in $U$ can be expressed as the inner product of characters of $S_b\times S_b$:
\begin{align*}
d_{\mu, \lambda}^{\nu}=\langle\chi^\mu\times\chi^\lambda,\,\chi\rangle = &\frac{1}{(b!)^2} \sum k_{g_{_1}}^2 \chi^{\mu}(g_1) \cdot \chi^{\lambda}(g_2)  \cdot |C_{S_b}(g_1)| \chi^{\nu}(g_1)\\
=&\frac{1}{(b!)^2} \sum k_{g_{_1}}^2 \chi^{\mu}(g_1) \cdot \chi^{\lambda}(g_1)  \cdot |C_{S_b}(g_1)| \chi^{\nu}(g_1) \\
=& \frac{1}{b!} \sum k_{g_{_1}} \chi^{\mu}(g_1) \cdot \chi^{\lambda}(g_1)  \cdot \chi^{\nu}(g_1)\\
=&c_{\mu,\lambda}^\nu, 
\end{align*}
the so called Kronecker coefficient, the multiplicity of $S^{\nu}$ in the $S_b$-module $S^{\mu} \otimes S^{\lambda}$.

Thus,
\begin{equation*}
		U \cong \bigoplus_{\mu,\lambda} c_{\mu, \lambda}^{\nu} S^{\mu} \times S^{\lambda},
\end{equation*}
which is equation \eqref{eq:U_2}
of Theorem~\ref{thm:main}.
Equation \eqref{eq:U_a} of Theorem~\ref{thm:main} follows now from Lemma~\ref{lem_arb:a}.
\end{proof}
Corollaries~\ref{cor:Foulkes} and \ref{cor:gen_Foulkes} follow from the observations
\begin{equation*}
	c_{\lambda, \lambda}^{(b)} = c_{\lambda^{\perp}, \lambda}^{(1^b)} = 1
\end{equation*}
and
\begin{equation*}
	b!=\sum_{\lambda \vdash b} \deg(S^{\lambda} \otimes S^{\lambda}) = \sum_{\lambda \vdash b}\deg(S^{\lambda^{\perp}} \otimes S^{\lambda})\leq  \deg( U_{(b), 2})= \deg( U_{(1^b), 2}) = b!.
\end{equation*}

\section{Consequences for the Generalized Foulkes Module}

\begin{lemma} \label{lem:tensorproducts}
The multiplicity of $S^{(1^b)} \times  S^{\mu}$ in $F_{\nu}^{a}$ is equal to the multiplicity of $S^{(1^b)} \times  S^{\mu}$ in $U_{\nu, a}$.
\end{lemma}

\begin{proof}
We need prove that $V_{\nu,a}^{\lambda}$ of \eqref{eq:Omega-indexing} has no summand $S^{(1^b)} \times S^{\mu}$ unless $\lambda=(1^b)$. 

As above, the basis of $V_{\nu, a}^{\lambda}$ is the set $B= B_{\nu, a}^{\lambda}=\{ e_{t_X} \mid X \in I\cap P_{\lambda},\ t \in\SYT(\nu) \}$. Observe that if $l_1$ and $l_2$ belong to the same part, say $X_k$, in $X = \{ X_1, X_2,\ldots,X_b\}$ then the transposition $(l_1\ l_2)$ fixes $e_{t_X}$, that is, $(l_1\, l_2)e_{t_X} = e_{t_X}$.

By contradiction, suppose that $M$ is a submodule of $V_{\nu, a}^{\lambda}$ isomorphic to $S^{(1^b)} \times S^{\mu}$. Let $m\ne0$ be an element of $M$. Then \[m = \sum_{\substack{X\in I\cap P_\lambda\\t\in \SYT(\nu)}}  c_{t_X}e_{t_X}.\]
Choose $t$ and $Y$ such that $c_{t_Y}\ne 0$. As $\lambda\ne (1^b)$, there exist $l_1, l_2 \leq b$ such that $l_1, l_2$ belong to the same part of $Y$. Let $e_{s_X}\in B\setminus\{e_{t_Y}\}$ be arbitrary. If $l_1, l_2$ belong to the same part of $X$ then $(l_1\, l_2)e_{s_X}=e_{s_X}\in B$. If $l_1,l_2$ are in different parts then $(l_1\, l_2) e_{s_X}=e_{s_{(l_1\, l_2)X}}=e_{s_Z}$ for some $Z\in P_\lambda$. However, this $Z$ and $Y$ cannot be in the same orbit. By 
Lemma~\ref{lem:arbpoly}, $e_{t_Y}=(l_1\,l_2)e_{t_Y}$ is not a summand of $(l_1\, l_2) e_{s_X}$ in either case. Therefore the coefficient of $e_{t_Y}$ in $(l_1\, l_2)m$ is also $c_{t_Y}$. However, $m\in M\cong S^{(1^b)} \times S^{\mu}$, so $(l_1\, l_2)m = -m$, which implies that the coefficient of $e_{t_Y}$ is $-c_{t_Y}\ne c_{t_Y}$, a contradiction. 
\end{proof}
\begin{corollary} \label{cor:Foulkestensor}
    The multiplicity of $S^{(1^b)} \times S^{\mu}$ in $F_\nu^a \downarrow_{S_b \times S_{ab -b}}$ is the same as the multiplicity of $S^\mu$ in $F_{\nu^\perp}^{a-1}$.
\end{corollary}
\begin{proof}
By Lemma~\ref{lem:tensorproducts}, the multiplicity of $S^{(1^b)} \times S^{\mu}$ in $F_\nu^{(a)} \downarrow_{S_b \times S_{ab-b}}$ is equal to its multiplicity in $U_{\nu, a}$. 

Let $c_{\mu, \lambda}^{\nu}$ be the Kronecker coefficient of $S^\nu$ in $S^\mu \otimes S^\lambda$. By Theorem~\ref{thm:main}, 
\begin{equation*}
    U_{\nu, a} = \bigoplus_{\mu,\lambda\vdash b} c_{\mu, \lambda}^{\nu} S^\mu \times F_\lambda^{a-1}.
\end{equation*}
Moreover $c_{\mu, \lambda}^{\nu} = c_{\lambda, \nu}^{\mu} = c_{\mu, \nu}^{\lambda}$. Therefore $c_{(1^b), \mu}^{\nu} = 0$ unless $\mu = \nu^\perp$, in which case $c_{(1^b), \nu^\perp}^{\nu} = 1$. 

From the above discussion we get that $S^{(1^b)} \times S^{\mu}$ can only be embedded in $S^{(1^b)} \times F_{\nu^\perp}^{a-1}$ among the above summands of $U_{\nu, a}$.   
\end{proof}
For a partition $\mu \vdash n$ of length $k$ and the Young subgroup  $S_{\mu}=S_{\mu_1} \times S_{\mu_2} ..\times S_{\mu_k}$, we define $M^{\mu}$ as the permutation module $1_{S_{\mu}} \uparrow^{S_n}$. In particular, $M^{(1^n)}$ is the regular module and $M^{(n)}$ is the trivial module. The following lemma is specific for the regular module.  
\begin{lemma}\label{lem:tabloids}
	Let $a, b \in \mathbb{N}$. Then
	\[\Inf_{S_b}^{S_a \wr S_b} M^{(1^b)} \uparrow^{S_{ab}} \cong M^{(a^b)}.\]
\end{lemma}
\begin{proof}
A basis of	$M^{\mu}$ is the set $\{\{t\} \mid t \ \text{is a }\mu\text{-tableau.}\}$. Choose $X \in I$, then a basis of $ \Inf_{S_b}^{S_a \wr S_b} M^{(1^b)} $ is the set $\{\{t_X\} \mid t \ \text{is a }(1^b)\text{-tableau}\}$. Therefore a basis for $\Inf_{S_b}^{S_a \wr S_b} M^{(1^b)} \uparrow^{S_{ab}}$ is the set $\{\{t_X\} \mid t \ \text{is a $(1^b)$ tableau}, X \in I\}$ which is also a basis for $M^{(a^b)}$.
\end{proof}
Now we are ready to derive the following corollary.
 \begin{corollary}\label{cor:gen_foulkes_min}
 	Let $\mu^\prime = \mu + (1^b)$. The multiplicity of $S^{\mu^\prime}$ in $F_{\nu^{\perp}}^{a + 1}$ is the same as the multiplicity of $S^{\mu}$ in $F_{\nu}^{a}$.
 \end{corollary}
\begin{proof}
From Corollary~\ref{cor:Foulkestensor}, we know that the multiplicity of $S^{(1^b)} \times S^{\mu}$ in $F_\nu^{a+1} \downarrow_{S_b \times S_{ab}}$ is the same as the multiplicity of $S^\mu$ in $F_{\nu^\perp}^{a}$. 
We also know that $S^{\nu}$ embeds into the regular module 
 $M^{(1^b)}$. Therefore 
	\begin{equation*}
		\Inf_{S_b}^{S_a \wr S_b} S^{\nu} \text{ embeds into }
  \Inf_{S_b}^{S_a \wr S_b} M^{(1^b)}.
	\end{equation*}
By the definition of the generalized Foulkes module and by Lemma~\ref{lem:tabloids},
	\begin{equation} \label{eq:Foulkes_in_M}
	 	F_{\nu}^a =\Inf_{S_b}^{S_a \wr S_b} S^{\nu} \uparrow^{S_{ab}}\text{ embeds into }
   \Inf_{S_b}^{S_a \wr S_b} M^{(1^b)} \uparrow^{S_{ab}}\cong M^{(a^b)}.
	\end{equation}

Since for all simple constituents $S^\lambda$ of $M^{(a^b)}$, $\lambda$ has at most $b$ parts, this also holds for $F_{\nu}^a$. 
By the Littlewood-Richardson Principle, \cite[Theorem 16.4]{GJamesRTSG}, all the constituents of $S^{(1^b)} \times S^{\mu} \uparrow^{S_{(a+1)b}}$ have more than $b$ parts except for $S^{\mu + (1^b)}$ which occurs with multiplicity $1$. 

\end{proof}
A generalized form of Corollary~\ref{cor:gen_foulkes_min} has been proven recently by de Boeck, Paget and Wildon \cite{RMM}. Their technique is much different.

\section{Concluding remarks and questions}
There have been many advances in the study of the Generalized Foulkes Module. One such result is given by Paget and Wildon\cite{RpMw}. They give a description of minimal and maximal Specht modules with respect to the dominance order on partitions.

Another result related to the Kronecker coefficients is given by Ikenmeyer, Mulmuley and Walter \cite{IMW}. They proved that deciding whether a Kronecker coefficient is zero is an $NP$-hard problem. B\"{u}rgisser and Ikenmeyer \cite{BI} showed that computing Kronecker coefficients is $\#P$-hard. Theorem~\ref{thm:main} gives a weak relationship between Kronecker coefficients and Foulkes modules which makes us wonder whether similar statements can be said for the multiplicity of $S^{\lambda}$ in $F_{\nu}^a$.
\begin{question}
What is the computational complexity of the coefficient of $S^\lambda$ in $F_\nu^a$?
\end{question}
There are many approaches to study the properties of $F_{\nu}^a$ but the problem of its decomposition into Specht modules still remains widely open. Studying the properties of the Generalized Foulkes module restricted to some small and large subgroups $G\leq S_n$, $F_{\nu}^a \downarrow_{G}$, might give some interesting information on the properties of $F_{\nu}^a$ and in turn help us understand its decomposition.

The generalized Foulkes module $F_{\nu}^a$ can be further generalized using a parameter $\lambda \vdash a$ to obtain $F_{\nu}^{\lambda}$. A description of such a generalization is given in \cite{RpMw}.
\begin{question}
Find the analogue of \eqref{eq:Omega-indexing} for $F_{\nu}^{\lambda} \downarrow_{S_k \times S_{n-k}}$.
Ideally, there should exist a distinguished component of $F_{\nu}^{\lambda} \downarrow_{S_k \times S_{n-k}}$ with description similar to the one in Theorem~\ref{thm:main}.
\end{question}

In fact, the decomposition \eqref{eq:Omega-indexing} of the restricted generalised Foulkes module  is the Mackey decomposition, see \cite[Lemma~8.7]{JLAlperin}. For more on this --- for the classical Foulkes module --- we refer to Definition~2.9 in \cite{GE}, the proof of Theorem~6.3 and the discussion before Lemma~6.8 in \cite{AHR}.
\section{Acknowledgement}
The present work was part of the PhD thesis of Sai Praveen Madireddi at Central European University. He would like to acknowledge the support he received.\\
Here we would also like to acknowledge the useful comments of Erzsébet Horváth to whose memory we dedicate this article.\\
P\'{a}l Heged\"{u}s was partially supported by Hungarian National Research, Development
and Innovation Office (NKFIH), Grant No.~138596. The project leading to
this application has received funding from the European Research Council
(ERC) under the European Union's Horizon 2020 research and innovation
programme, Grant agreement No.~741420
   
\bibliographystyle{is-unsrt}
\bibliography{main}
  
\end{document}